\documentclass{amsart}
\usepackage{graphicx}
\usepackage{newlattice}

\theoremstyle{plain}
\newtheorem{theorem}{Theorem}
\newtheorem*{theoremn}{Theorem}
\newtheorem{lemma}[theorem]{Lemma}

\newtheorem{corollary}[theorem]{Corollary}

\newcommand{\swing}{\mathbin{\raisebox{2.0pt}
       {\rotatebox{160}{$\curvearrowleft$}}}}
\newcommand{\length}{\tup{length}}

\begin{document}

\title[Some applications of the Cz\'edli-Schmidt Sequences]{Applying the Cz\'edli-Schmidt Sequences
to~congruence properties of planar semimodular~lattices}

\author{G. Gr\"{a}tzer} 
\email[G. Gr\"atzer]{gratzer@me.com}
\urladdr[G. Gr\"atzer]{http://server.maths.umanitoba.ca/homepages/gratzer/}

\date{April 6, 2020}

\begin{abstract}
Following G.~Gr\"atzer and E.~Knapp, 2009,
a planar semimodular lattice $L$ is \emph{rectangular},
if~the left boundary chain
has exactly one doubly-irreducible element, $c_l$,
and the right boundary chain  
has exactly one doubly-irreducible element, $c_r$,
and these elements are complementary.

The Cz\'edli-Schmidt Sequences, introduced in 2012, 
construct rectangular lattices.
We use them to prove some structure theorems. 
In particular, we prove that for a slim
(no $\mathsf{M}_3$ sublattice) rectangular lattice~$L$,
the congruence lattice $\Con L$
has exactly $\length[c_l,1] + \length[c_r,1]$ 
dual atoms and a dual atom in $\Con L$
is a congruence with exactly two classes.
We also describe the prime ideals in a slim rectangular lattice. 
\end{abstract}

\maketitle

\section{Introduction}\label{S:intro}

\subsection{The Cz\'edli-Schmidt Sequences}
\label{S:SC}

G.~Cz\'edli and E.\,T. Schmidt~\cite{CS12}
proved the Structure Theorem for 
Slim Rectangular Lattices, according to which 
every slim rectangular lattice can be constructed
from a planar distributive lattice, a grid,
with the Cz\'edli-Schmidt Sequences, 
see Section~\ref{S:Structure} for the definitions.
In this paper, we find new applications 
for the Cz\'edli-Schmidt Sequences.

\subsection{Congruence lattices of SPS lattices}\label{S:SPS}

The topic of this paper started 
in G. Gr\"atzer, H.~Lakser, 
and E.\,T. Schmidt \cite{GLS95}, 
where we proved that every finite distributive lattice $D$ 
can be represented as the congruence lattice 
of a PS (planar semimodular) lattice $L$. 
The sublattices $\SM3$
played a crucial role in the construction of $L$,  
so we asked
(\cite[Problem 1]{gG16} and \cite[Problems 4.7--4.10]{CFL2}) 
what happens if we only consider SPS lattices
(Slim PS, where ``slim''means that 
there is no $\SM 3$ sublattice)?

\subsection{The Two-cover Theorem}\label{S:Two}

In~\cite[Theorem 5]{gG16}, 
I proved the Two-cover Theorem:
The congruence lattice of an SPS lattice 
has the property 
\begin{enumeratei}
\item[(2C)] Every join-irreducible 
congruence has at most two join-irreducible covers 
(in~the order of join-irreducible congruences).
\end{enumeratei}
See also \cite[Theorem 25.2]{CFL2},
G. Cz\'edli \cite[Theorem~1.1]{gC14}, 
and my paper~\cite{gG19}. 

G. Cz\'edli \cite[Theorem 1.1]{gC14} 
proved that the converse is false 
by exhibiting an eight-element 
distributive lattice, $\SD 8$ (see Figure~\ref{F:D8}), satisfying (2C), which cannot be represented
as the congruence lattice of an SPS lattice;
see also my paper \cite{gG15}.
\begin{figure}[htb]
\centerline{\includegraphics{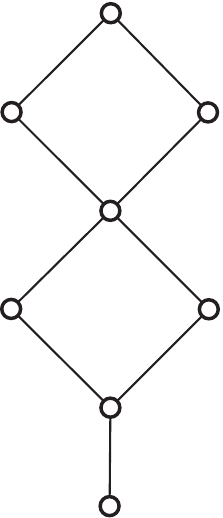}}
\caption{The distributive lattice $D_8$}
\label{F:D8}
\end{figure}
In \cite{gG19}, I observed that
the three-element chain $\SC3$ cannot be represented 
either as the congruence lattice of~an SPS lattice.
This paper is the start of the present one.

\subsection{The main result}\label{S:main}

Following G.~Gr\"atzer and E.~Knapp~\cite{GKn09},
a planar semimodular lattice (by definition, finite) 
$L$ is \emph{rectangular},
if~the left boundary chain
has exactly one doubly-irreducible element, $c_l$,
and the right boundary chain  
has exactly one doubly-irreducible element, $c_r$,
and these elements are complementary.
Let $\SB n$ denote the Boolean lattice with $n$ atoms.
 
In this paper, we use the Cz\'edli-Schmidt Sequences
to prove the following result.

\begin{theorem}\label{T:main}
Let $L$ be a slim rectangular lattice $L$
and let 
\[
   t = \length[c_l,1] + \length[c_r,1].
\]
Then the congruence lattice $\Con L$
has exactly $t$ dual atoms and a dual atom in~$\Con L$
is a congruence with exactly two classes. 
\end{theorem}

Since $\Con L$ is distributive, 
we obtain the following statement.

\begin{corollary}\label{C:xx}
Let $L$ be a slim rectangular lattice. 
Then  $\Con L$ has a filter isomorphic 
to the Boolean lattice $\SB t$.
\end{corollary}

On the way to proving Theorem~\ref{T:main},
we describe the prime ideals 
of a slim rectangular lattice $L$,
following up an observation in \cite{gG19}.
We shall also discuss variants of Theorem~\ref{T:main}
for rectangular lattices, PS lattices, and SPS lattices.

\subsection{Notation}\label{S:Notation}

For the basic concepts and notation,
see my books \cite{LTF} and \cite{CFL2}.

\subsection{Outline}\label{S:Outline}

We recall some easy facts about slim rectangular lattices
in Section~\ref{S:Background}
as well as we state the Structure Theorem and the Swing Lemma.
 
In Section~\ref{S:preliminary}, 
we prove some preliminary results
on slim rectangular lattices.
We~describe the prime ideals of a
slim rectangular lattice in Section~\ref{S:Prime}.
We~investigate in Section~\ref{S:structure}
how adjacent congruence classes interface.
A prime ideal $P$ of a lattice~$L$ is naturally associated
with a congruence~$\bgp(P)$, which we call a prime congruence.
In Section~\ref{S:Primecong} we prove
that a dual atom in $\Con L$ 
of a slim rectangular lattice $L$ 
is a~prime congruence. 
The main result of this paper follows.

\subsection*{Acknowledgement}
The author thanks the referee for many valuable comments.

\section{Background}\label{S:Background}

\subsection{Some known results}\label{S:known}

We will use the two results in the next lemma,
implicitly or explicitly.
 
\begin{lemma}\label{L:known}
Let $L$ be an SPS lattice. Then the following statements hold.
\begin{enumeratei}
\item An element of $L$ has at most two covers.
\item Let $x \in L$ 
cover three distinct elements $u$, $v$, and $w$.
Then the set $\set{u,v,w}$ gene\-rates an $\SfS 7$ sublattice
\lp see Figure~\ref{F:N7}\rp.
\end{enumeratei}
\end{lemma}

See my paper \cite{gG15} and 
G. Cz\'edli and G. Gr\"atzer \cite{CG16}
for some proofs and references.

\begin{figure}[htb]
\centerline{\includegraphics{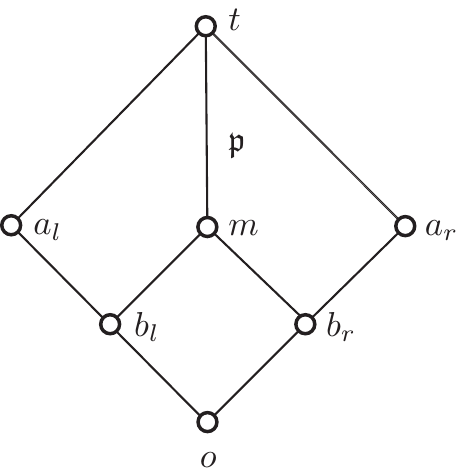}}
\caption{The lattice $\SN7$}
\label{F:N7}
\end{figure}

As introduced in O. Ore \cite{oO43}, 
see also S. MacLane \cite{sM43},
a \emph{cell} $A$ in a planar lattice 
consists of two chains $C$ 
\tup{(}with zero $0_C$ and unit $1_C$\tup{)} 
and $D$ \tup{(}with zero $0_D$ and unit $1_D$\tup{)} 
such that the following conditions hold:
\begin{enumeratei}
\item $0_C  =  0_D$ and $1_C  =  1_D$;
\item $C$ and $D$ are maximal 
in $[0_C, 1_C] = [0_D, 1_D]$;
\item every $x \in C-\set{0_C, 1_C}$ 
is to the left of every $y \in D - \set{0_D, 1_D}$;
\item there are no elements inside the region 
bounded by $C$ and $D$.
\end{enumeratei}

A \emph{$4$-cell} is a cell with $|C| = |D| = 3$. 
A \emph{$4$-cell lattice} is a lattice 
in which all cells are $4$-cells.

For the following observation see G. Gr\"atzer and E. Knapp~\cite[Section 4]{GKn08}.

\begin{lemma}\label{L:known2}
A PS lattice is a $4$-cell lattice.
\end{lemma}
 
The following statement, 
see G. Gr\"atzer and E. Knapp~\cite[Lemma 4]{GKn09},
plays an important role.

\begin{lemma}\label{L:known3}
In a slim rectangular lattice, 
the bottom boundaries are ideals 
and the upper boundaries are filters.
\end{lemma}

\begin{corollary}\label{C:rightboundary}
Let $L$ be a slim rectangular lattice. 
Then for every $x \in L$, the element~$x \jj c_r$ 
is in the upper right boundary of $L$, 
and symmetrically.
\end{corollary}

\begin{proof}
Indeed, by Lemma~\ref{L:known3}, 
the upper right boundary of $L$ 
is the filter generated by~$c_r$.
Since $x \jj c_r \in \fil {c_r}$,
it follows that $x \jj c_r$ is in the upper right boundary.
\end{proof}

\begin{figure}[ht] 
\centerline{\includegraphics{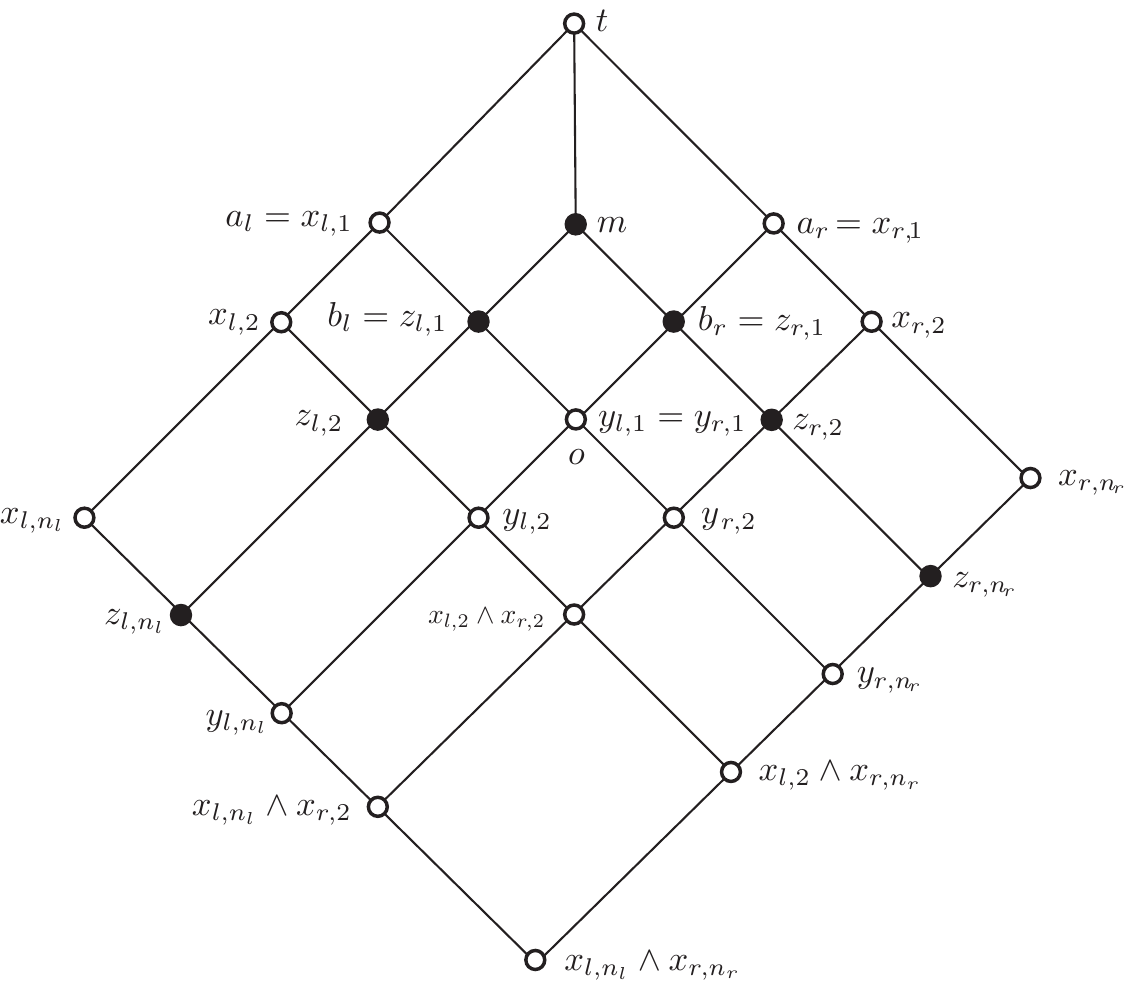}}
\caption{Notation for the fork insertion}\label{F:forknotation}
\end{figure}

\subsection{The Structure Theorem}\label{S:Structure}

Let $L$ be a slim rectangular lattice.
A \emph{Cz\'edli-Schmidt Sequence} for $L$
is a~sequence of slim rectangular lattices and 
a sequence of covering squares:
\begin{equation}\label{E:FES}
   \begin{split}
       &D = L_1, L_2, \dots, L_s = L,\\ 
       &S^1 = \set{o^1, a_l^1, a_r^1, t^1},
       S^2 = \set{o^2, a_l^2, a_r^2, t^2}, \dots,
       S^{s-1} = \set{o^{s-1}, a_l^{s-1}, a_r^{s-1}, t^{s-1}},
   \end{split}
\end{equation}
where $S^i$ is a covering square in $L_{i}$ and
we obtain $L_{i+1}$ from $L_{i}$ 
by~inserting a fork at $S^{i}$ 
(in formula, $L_{i+1} = L_{i}[S^{i}]$)
for $i = 1, \dots, s-1$. 
 
For detailed descriptions of the fork extension, see
G.~Cz\'edli and E.\,T. Schmidt~\cite{CS12}, 
G. Gr\"atzer~\cite{CFL2}, and other papers in the references.

We use the standard notation
for fork insertions, see Figure~\ref{F:forknotation}
(where the black filled elements represent 
the inserted elements). 

The following result is G.~Cz\'edli 
and E.\,T. Schmidt~\cite[Lemma 22]{CS12}.

\begin{theoremn}[The Structure Theorem for 
Slim Rectangular Lattices]\label{T:ES}
Let $L$ be a \emph{slim rectangular lattice}.
Then there is a grid $D = \SC p \times \SC q$,
where $p,q \geq 2$, and a~Cz\'edli-Schmidt Sequence
from $D$ to $L$.
\end{theoremn}

Note that the integer $s$ in \eqref{E:FES} is an invariant.

We call $D$ the \emph{grid} of $L$; 
it is isomorphic to a sublattice of $L$.
\subsection{The Swing Lemma}\label{S:Swing}
For the prime intervals $\fp, \fq$ of an SPS lattice $L$, 
we define a binary relation:
$\fp$~\emph{swings} to $\fq$, 
if $1_\fp = 1_\fq$, 
this element covers at least three elements,
and $0_\fq$ is neither the left-most nor the right-most element
covered by~$1_\fp = 1_\fq$, see Figure~\ref{F:n7+}. 

\begin{figure}[thb]
\centerline{\includegraphics{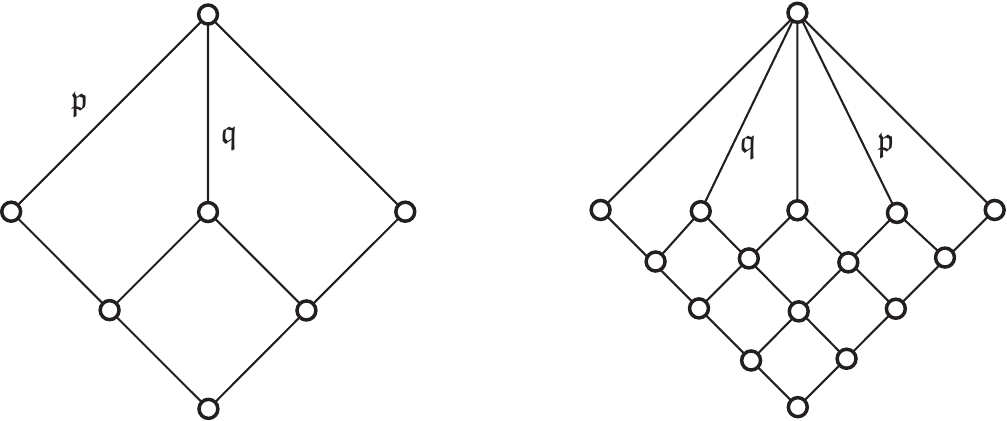}} 
\caption{Swings, $\fp \protect\swing \fq$}\label{F:n7+}
\end{figure}

The following result is from my paper \cite{gG15}.

\begin{lemma}[Swing Lemma]\label{L:SPSproj}
Let $L$ be an SPS lattice 
and let $\fp$ and $\fq$ be distinct prime intervals in $L$. 
If $\fq$ is collapsed by $\con{\fp}$,
then there exists a prime interval~$\fr$ 
and sequence of pairwise distinct prime intervals
\begin{equation}\label{Eq:sequence}
\fr = \fr_0, \fr_1, \dots, \fr_n = \fq
\end{equation}
such that $\fp$ is up perspective to $\fr$, and 
$\fr_i$ is down perspective to or swings to $\fr_{i+1}$
for $i = 0, \dots, n-1$. 
In addition, the sequence \eqref{Eq:sequence} also satisfies 
\begin{equation}\label{E:geq}
   1_{\fr_0} \geq 1_{\fr_1} \geq \dots 
      \geq 1_{\fr_n} = 1_{\fq}.
\end{equation}
\end{lemma}

The Swing Lemma is easy to visualize. 
Perspectivity up is ``climbing up'', 
perspectivity down is ``sliding''. 
So we get from $\fp$ to $\fq$ by climbing up once
and then alternating sliding and swinging.

\section{Some preliminary results
on slim rectangular lattices}\label{S:preliminary}

In this section, we prove some elementary results
about slim rectangular lattices. 
Let $L$ be a slim rectangular lattice
with the Cz\'edli-Schmidt Sequence~\eqref{E:FES} and
with the grid $D = \SC p \times \SC q$.

Let $c_l$ and $c_r$ be the corners of $D$,
and let $c_l^i$ and $c_r^i$ be the corners of $L_i$
for $i = 1, \dots, s-1$. 

We prove the next two lemmas utilizing the
Cz\'edli-Schmidt Sequences.

\begin{lemma}\label{L:corners}
$c_l = c_l^i$ and $c_r = c_r^i$ for $i = 1, \dots, s$.
\end{lemma}

\begin{proof}
By induction on $s$ as in \eqref{E:FES}.
By definition, $c_l = c_l^1$ and $c_r = c_r^1$.
Assume that the statement holds for~$s-1$.
We obtain $L_{s}$ from $L_{s-1}$ 
by adding a fork at $S^{s-1}$, see Figure~\ref{F:forknotation},
so there is only one new element on the left boundary, 
and it is a meet-reducible element below $c_l = c_l^{s-1}$.
Therefore, $c_l$ is the only doubly irreducible element
on the left boundary of $L_{s-1}$, and so $c_l = c_l^{s}$. 
Similarly, $c_r = c_r^s$. 
\end{proof}

\begin{corollary}\label{C:upper}
Let $L$ be a slim rectangular lattice 
and let $S$ be a covering square in~$L$.
Then the upper left boundaries of $L$ and $L[S]$ 
are the same \lp and symmetrically\rp.
Therefore, the chains $\SC p$ and $\SC q$  
are isomorphic to $[c_l,1]$ and $[c_r,1]$, respectively.
\end{corollary}

\begin{corollary}\label{C:base}
For a slim rectangular lattice $L$,
the grid is unique up to isomorphism.
\end{corollary}

\begin{lemma}\label{L:xx}
Let $L$ be a slim rectangular lattice. 
Then 
\begin{align}
  \length[0, c_l^s]  = \length[0, c_l] + s - 1,\\
  \length[0, c_r^s]  = \length[0_r, c_r] + s -1.
\end{align}
\end{lemma}

\begin{proof}
Indeed, each step in \eqref{E:FES} adds an element 
to the lower boundary chains.
\end{proof}

It now follows that 
\begin{equation}
   \length[0, c_l] - \length[0, c_r]
   = \length[c_r, 1] - \length[c_l, 1].
\end{equation}
This immediately follows also from semimodularity.

\section{Prime ideals}\label{S:Prime}

We describe the prime ideals of a
slim rectangular lattice in this section.

The two lemmas of this section are proved 
using the Cz\'edli-Schmidt Sequences.

Let $L$ be a planar semimodular lattice.
We call the element $m \in L$ a \emph{middle} element 
of $L$ if there is an $\SN7$ sublattice such that 
$m$ is the middle element of the $\SN7$ sublattice. 

\begin{lemma}\label{L:four}
Let $L$ be a slim rectangular lattice.
Let $a$ be an element of $L$. 
Then one of the following statements holds:
\begin{enumeratei}
\item the element $a$ is 
on the upper boundary of $L$;
\item the element $a$ is meet-reducible;
\item the element $a$ is a middle element.
\end{enumeratei}
\end{lemma}

\begin{proof}
By induction on $s$ as in \eqref{E:FES}. 
If $s = 1$, then $L = D$,
and the statement holds for a grid.
Let the statement hold for $s - 1$.  
The new elements of $L_{s}$, see Figure~\ref{F:forknotation},
form the set $[z_{l, n_l}, m] \uu [z_{r, n_r}, m]$, 
and they consist of the element $m$---satisfying (ii)---or 
an element in the set 
$[z_{l, n_l}, b_l] \uu [z_{r, n_r}, b_r]$,
all of which are meet-reducible, so satisfying (ii). 
\end{proof}

\begin{lemma}\label{L:prime1}
Let $L$ be a slim rectangular lattice
and let $p \in L$.
If $p \neq 1$ and $p$ is 
in~the upper left boundary of $L$,
then there exists an element $q$ 
in~the lower right boundary of $L$,
so that $\set{\id{p}, \fil{q}}$ is a partition of $L$.
\end{lemma}

\begin{proof}
By induction on $s$ as in \eqref{E:FES}.
If $s = 1$, then $L = D$,
and the statement holds for a grid with $q = p \mm c_r$.
Let the statement hold for $s - 1$,
and therefore, for $L_{s-1}$.
So let $p \neq 1$, 
let $p$ be on~the upper left boundary of $L_{s}$ 
(or symmetrically).  
Recall that by Corollary~\ref{C:upper},
the upper left boundaries of $L$ and $L_{s-1}$ 
are the same.
Let $q_{s-1}$ be the element 
in~the lower right boundary of $L_{s-1}$
that exits by the induction hypothesis
and let $S = S^{s-1}$ be the covering square of $L_{s-1}$.
We use the notation: 
\[W = [m, z_{r, n_r}] \uu[m, y_{l, n_l}].\]

There are three cases to distinguish.

\emph{Case 1.} $S \ci \filsub{L_{s-1}}{q_{s-1}}$, 
as illustrated in Figure~\ref{F:S+123}.
Then 
\[
  W \ci \filsub{L_{s}}{q_{s-1}} \uu \idsub{L_{s}}{p},
\] 
therefore, $\set{\id{p}, \fil{q}}$ is a partition of $L$ with $q = q_{s-1}$. 

\begin{figure}[htb]
\centerline{\includegraphics[scale=.75]{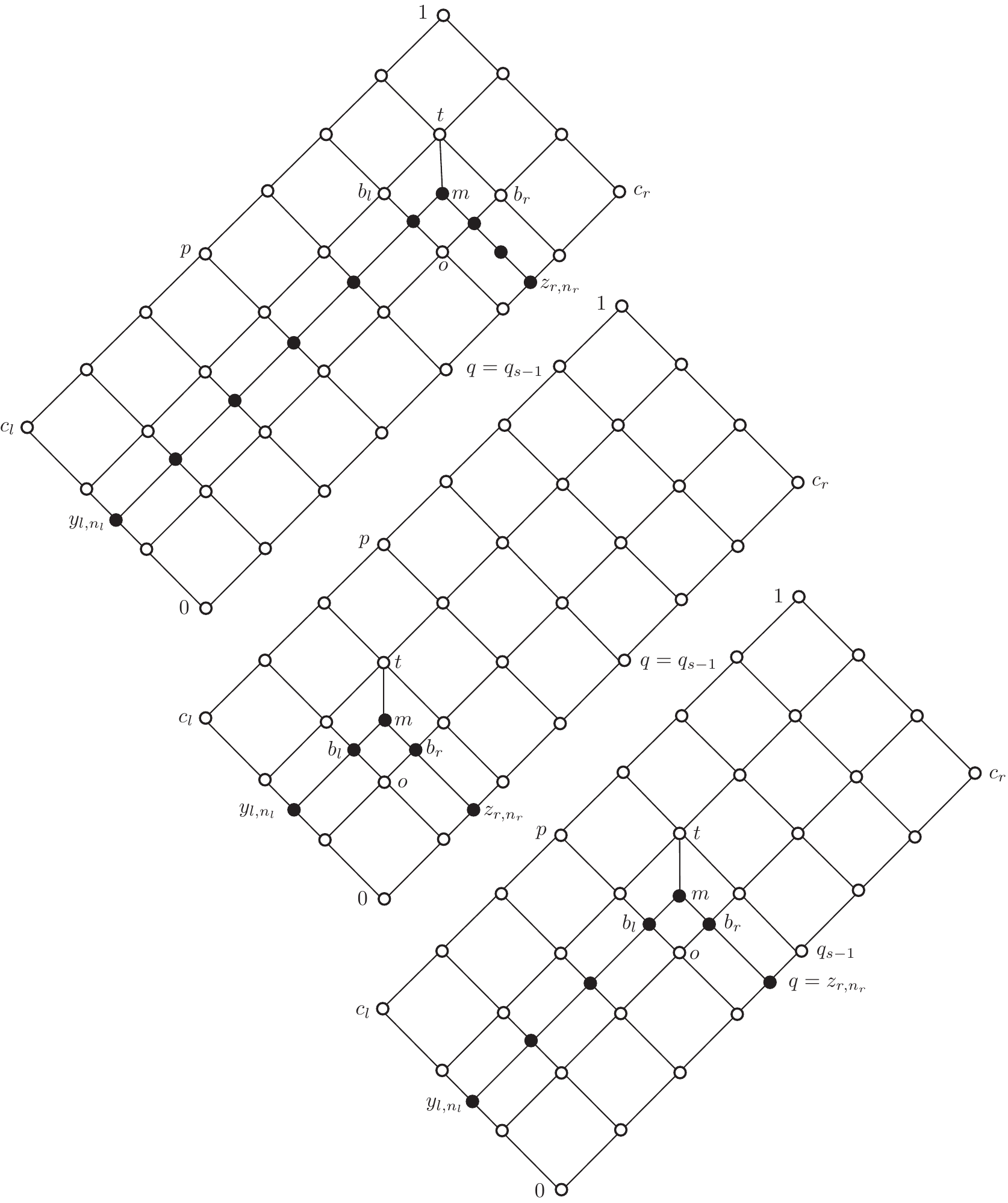}}
\caption{Proof of Lemma~\ref{L:prime1}}
\label{F:S+123}
\end{figure}

\emph{Case 2.} $S \ci \idsub{L_{s-1}}{p}$, 
as illustrated in Figure~\ref{F:S+123}.
In this case, 
\[
  W \ci \idsub{L_{s}}{p},
\] 
so $\set{\id{p}, \fil{q}}$ is a partition of $L$ with $q = q_{s-1}$. 

\emph{Case 3.} $S \nci \filsub{L_{s-1}}{q_{s-1}}, 
\idsub{L_{s-1}}{q_{s-1}}$, 
also illustrated in Figure~\ref{F:S+123}.
In this case, the two elements 
on the right upper boundary of $S$
are in $\filsub{L_{s-1}}{q_{s-1}}$ 
and the other two elements are in $\idsub{L_{s-1}}{p}$.
The newly inserted elements in $[m, y_{l, n_l}]$ 
are in $\idsub{L_{s-1}}{p}$, 
and the rest of them, $[m, z_{r, n_r}]$, 
are in $\filsub{L}{z_{r, n_r}}$,
so $\set{\id{p}, \fil{q}}$ 
is a partition of $L$ with $q = z_{r, n_r}$.
Note that $p \mm c_r \prec q \prec q_{n-1}$.
\end{proof}

\begin{corollary}\label{C:prime1}
Let $L$ be a slim rectangular lattice and let $p \in L$.
If $p \neq 1$ and $p$ is in~the upper boundary of $L$,
then the ideal $P = \id p$ of $L$ is prime.
\end{corollary}

\begin{proof}
By Lemma~\ref{L:prime1} and its symmetric counterpart.
\end{proof}

A very special case of this result 
was found in G. Gr\"atzer~\cite{gG19}.
In a sense, this paper was the starting point
of the present one.

\begin{theorem}\label{T:prime}
Let $L$ be a slim rectangular lattice
and let $1 \neq a \in L$. 
Then $P = \id a$ is a prime ideal of $L$
if and only if $a$ is in~the upper left or 
upper right boundary of~$L$.
\end{theorem}

\begin{proof}
Since $\id p$ is not a prime 
either for a meet reducible $p$ 
or for a middle element
$p = m$ (because $m > a_l \mm a_r$) as in Figure 2,  
Lemma~\ref{L:four} applies.
\end{proof}

\section{The structure of congruence classes}\label{S:structure}

I have known for a long time  
how adjacent congruence classes interface in a lattice.
In this section, I prove two of these results,
because they will be needed in Section~\ref{S:Primecong}.
The first lemma is related to some discussions
in G. Cz\'edli \cite{gC96} and \cite{gC09}.

\begin{lemma}\label{L:congclass}
Let $\bga$ be a congruence of a lattice $L$ 
and let $A = [0_A, 1_A]$ and $B = [0_B, 1_B]$ 
be congruence classes of $\bga$
satisfying that $A \prec B$ in $L/\bga$.
Then for every $x \in A$, 
there is a smallest $x^B \in B$ with $x \leq x^B$
and for every $x \in B$, 
there is a greatest $x_A \in A$ with $x \geq x_A$.
Moreover, $(x^B)_A \prec x^B$ for every $x \in A$.
\end{lemma}

\begin{proof}
Define $x^B = x \jj 0_B$ for $x \in A$ and 
$y_B = y \mm 1_A$ for $y \in B$.
This sets up a~standard Galois connection, 
so only the last statement needs proof.
Let us assume that $(x^B)_A < u < x^B$.
By the definition of $x_B$, 
it follows that $u \nin B$ and similarly, $u \nin A$.
Therefore, $A < u/\bga < B$ in $L/\bga$,
contrary to the assumption that $A \prec B$ in $L/\bga$.
\end{proof}

\begin{figure}[htp]
\centerline{\includegraphics{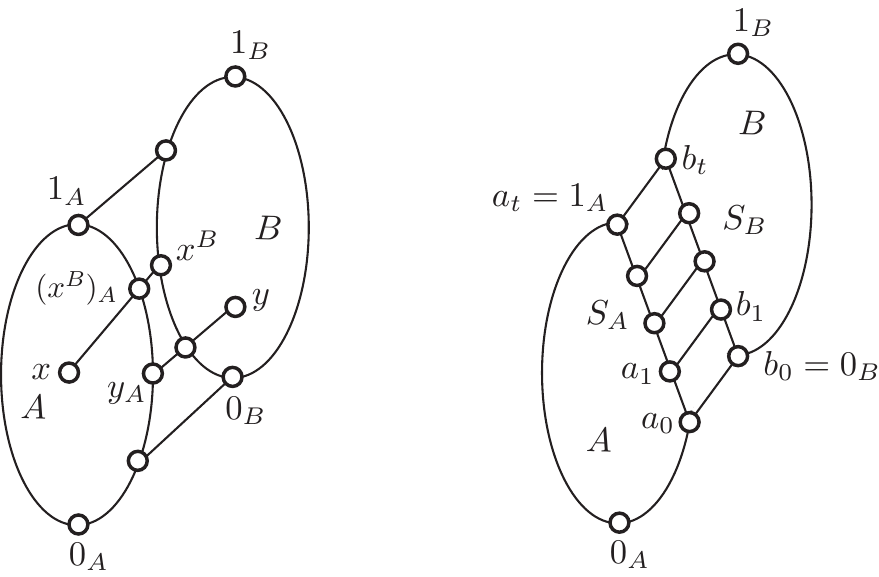}}
\caption{Two illustrations of $A \prec B$ in $L/\bga$}
\label{F:edge}
\end{figure}

\begin{lemma}\label{L:congclass2}
Let $L$ be a slim, planar, semimodular lattice. 
Let $\bga$ be a congruence of $L$ 
and let $A, B$ be congruence classes of $\bga$
satisfying that $A \prec B$ in $L/\bga$.
Then there is a maximal chain 
\[
   S_A =\set{1_A \mm 0_B = a_0 \prec a_1 \prec \dots \prec a_t = 1_A}
\]
on the right boundary of $A$ and
there is a maximal chain 
\[
   S_B = \set{0_B = b_0 \prec b_1 \prec \dots \prec b_t 
       = 1_A \jj 0_B}
\]
on the left boundary of $B$---or symmetrically.
The chain $S_A$ is isomorphic to $S_B$ by the map 
$\gf_A \colon x \mapsto x \jj 0_B$;
the inverse isomophism is  
$\gf_B \colon x \mapsto x \mm 0_B$.
\end{lemma}

\begin{proof}
If the elements $1_A$ and $0_B$ are comparable,
then $1_A < 0_B$ and the statement is true
with the singletons $S_A$ and $S_B$.
So we can assume that $1_A$ and $0_B$ are incomparable. 
By symmetry, 
we can also assume that $0_B$ is to the right of $1_A$.

Let $a_0 = 1_A \mm 0_B$ and $b_0 = 0_B$. 
If $a_0 = 0_B$, then $a_0 \in A \ii B$, a contradiction, 
since $A \prec B$ in $L/\bga$, 
so $A$ and $B$ are disjoint. 
Hence, $a_0 < b_0$.

We claim that $a_0 \prec b_0$.
Indeed, let there be an element $z$ of $L$
with $a_0 < z < b_0$.
If $z \in A$, then $z \leq 1_A$, 
so $z = a_0$, a contradiction.
If $z \in B$, 
then $0_B$ is not the smallest element of $B$,
a contradiction.
Therefore,  $A < z/\bga < B$ in $L/\bga$,
contradicting the assumption that  
$A \prec B$ in $L/\bga$. 
This verifies the claim.
  
By David Kelly and Ivan Rival \cite{KR75},
this implies that
$a_0$ is on the boundary of $A$, say,  
on the right boundary.
This allows us to take a maximal chain
\[
   S_A = \set{a_0 \prec a_1\prec \dots \prec a_t=1_A}
\]
of $[a_0,1_A]$ on the right boundary of $A$. 
Put $b_i = a_i \jj b_0$ for $i=0, \dots, t$.
Since $A \jj B = B$ in $L/\bga$,   
we get that $b_i \in B$ for $i=0, \dots, t$. 
So $a_i<b_i$. 
By~semimodularity, $a_i\prec b_i$ for $i=0, \dots, t$. 
Again, by semimodularity, 
we obtain that $b_i\preceq b_{i+1}$.
Since $1_A=a_t\prec b_t$, 
we can see that
\[
   \set{a_0 \prec a_1\prec \dots \prec a_t = 
      1_A \prec b_t}\]
is a maximal chain 
in the interval $[a_0,b_t]$ of length is $t$.
The chain 
\[\set{a_0\prec b_0 \preceq b_1 \preceq \dots 
   \preceq b_t}
\] 
is a maximal chain in the same interval, 
so by the Jordan-H\"older property
of finite semimodular lattices, 
we obtain that it is also of length $t$.
Now it follows that 
\[
   S_B= \set{b_0 \prec b_1 \prec \dots \prec b_t}
\]
satisfies the requirements of the lemma, 
since all the squares depicted 
on the right of Figure~\ref{F:edge} are covering squares. 
\end{proof}

We call $S_A \times \SC 2$ the \emph{ladder} 
associated with $A \prec B$. 
Note that it has a single rung if $1_A < 0_B$ 
(equivalently, if $1_A \prec 0_B$).

\section{Prime congruences}\label{S:Primecong}

A congruence $\bgp$ of a lattice $L$ is \emph{prime}
if it has exactly two blocks. 
Clearly, one of its blocks is a prime ideal $P$.
Since $P$ determines $\bgp$,
we use the notation $\bgp(P)$ for $\bgp$. 
Every prime congruence of $L$
is a dual atom in $\Con(L)$.
Also, if a congruence has only two congruence classes,
then it is prime.

\begin{theorem}\label{T:dualatom}
Let $L$ be a slim rectangular lattice 
and let the congruence $\bgp$ of $L$ 
be a dual atom in $\Con L$. 
Then the congruence $\bgp$ is prime.
\end{theorem}

\begin{proof}
Let $\bgp$ be a dual atom in $\Con L$.
Let $\bgp$ partition the upper left boundary
into $b$ blocks.

\emph{Case 1}: $b = 1$. 
Equivalently, $\cng c_l = 1 (\bgp)$.
Meeting both sides with $c_r$,
we obtain that $\cng 0 = c_r (\bgp)$.
By Corollary~\ref{C:rightboundary},
for every $x \in L$, the element~$x \jj c_r$ 
is in the upper right boundary of $L$,
so $\cng x=x \jj c_r(\bgp)$. 
Thus we can choose a subchain $C$ of $[c_r,1]$
with the property that every congruence class
of $\bga$ contains exactly one element of $C$.
By the First Isomorphism Theorem *
(see, for instance, \cite[Exercise I.3.61]{LTF}),
we have the isomophism $L/\bgp \iso C$,
so $L/\bgp$ is a chain.
Since the congruence $\bgp$ of $L$ 
is a dual atom in $\Con L$,
by the Second Isomorphism Theorem 
((see, for instance, \cite[Theorem 220]{LTF}))
the lattice $L/\bgp$ is simple. 
A simple distributive lattice has two elements, 
so $\bgp$ is prime, as required.

\emph{Case 2}: $b = 2$. 
Equivalently, there is a prime interval $\fp$ 
on the upper left boundary of $L$, such that
$\cng c_l=0_\fp (\bgp)$, $\cng 1_\fp=1(\bgp)$,
and $\ncng 0_\fp=1_\fp(\bgp)$, or symmetrically.

For $c_l = 0_\fp$ or $1_\fp = 1$ or both, 
define $c_l = 0_\fp$. Then $L/\bgp \iso Q/\bgp$. 

Otherwise, $c_l <  0_\fp \prec 1_\fp < 1$.

We use the ladder of Lemma~\ref{L:congclass2}, 
see Figure~\ref{F:edge}. 
Let $q$ be the cover of $0_\fp \mm c_r$ 
on the lower right boundary.
Note the ideal $P = [0, 0_\fp]$
and the filter $Q = [q , 1]$. 
The two sides of the ladder are
\begin{align*}
   S_A &=\set{0_\fp \mm c_r = a_0 \prec \dots 
       \prec a_t = 0_\fp},\\
   S_B &= \set{0_B = q  \prec \dots \prec b_t 
       = 1_\fp},
\end{align*}
using the notation of Figure~\ref{F:PQ}.
The chain $S_A$ is shaded black and 
the chain $S_B$ is shaded gray.
\begin{figure}[htb]
\centerline{\includegraphics{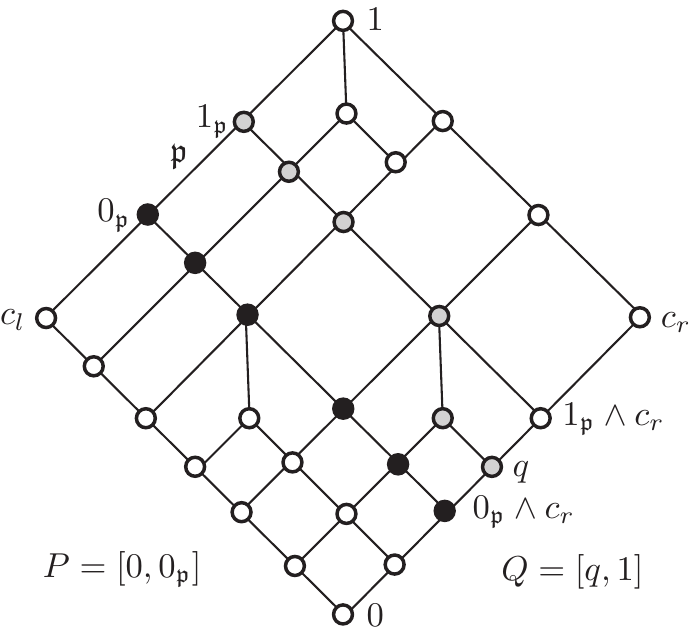}}
\caption{Notation for the proof of 
Theorem~\ref{T:dualatom}, Case 2; 
the chain $S_A$ is shaded black
and the chain $S_B$ is shaded gray}
\label{F:PQ}
\end{figure}

We argue as in Case 1, \emph{mutatis mutandis},
that for every element $x \in P$, 
there is an element $y \in S_A$ 
such that $\cng x=y(\bgp)$. 
The same way, for every element $x \in [1_\fp \mm c_r, 1]$, 
there is an element $y \in [c_r, 1]$ 
such that $\cng x=y(\bgp)$.

Since the corresponding prime intervals 
of $S_A$ and $S_B$ are perspective
(as illustrated by Figure~\ref{F:edge}),
it follows that $S_A/\bgp$ and $S_B/\bgp$
are isomorphic. 
Therefore, $L/\bgp$ can be obtained by gluing together
$[0, 0_\fp]/\bgp$ and $[q, 1]/\bgp$ over a chain
$S_A/\bgp \iso S_B/\bgp$. 
Both lattices $[0, 0_\fp]/\bgp$ and $[q, 1]/\bgp$
are slim rectangular lattices so their gluing
over $S_A/\bgp \iso S_B/\bgp$ 
is also a slim rectangular lattice
by G. Gr\"atzer and E. Knapp~\cite[Lemma 5]{GKn09}.

Since $[0,0_p]/\bgp$ and $[q,1]/\bgp$ 
are isomorphic to the chains
$C_A/\bgp$ and $C_B/\bgp$, respectively, 
the lattices $[0,0_p]/\bgp$ and $[q,1]/\bgp$ 
are distributive. 
Gluing these two lattices over 
$S_A/\bgp \cong S_B/\bgp$,  
the Second Isomorphism Theorem
gives again that $L/\bgp$ 
is a simple distributive lattice and so $\bgp$ is prime,
and the statement follows.

\emph{Case 3}: $b \geq 3$. 
Equivalently, there are prime intervals $\fp$ and $\fq$
on the upper left boundary of $L$ (or symmetrically)
such that $1_\fp < 0_\fq$, $\ncng 0_\fp = 1_\fp (\bgp)$,
and $\ncng 0_\fq = 1_\fq (\bgp)$. 
Since $\bgp$ is a dual atom in $\Con L$, 
it follows that
\[
   \bgp \jj \con \fp = \bgp \jj \con \fq = \one.
\]
Therefore, $\con \fp \leq \bgp \jj \con \fq$.
Since $\fp$ is a prime interval, we get that 
$\con \fp \leq \bgp$ or $\con \fp \leq \con \fq$.
The inequality $\con \fp \leq \bgp$ 
contradicts the assumption 
that $\ncng 0_\fp = 1_\fp (\bgp)$, 
so we conclude that $\con \fp \leq \con \fq$ holds.

By the Swing Lemma (Lemma~\ref{L:SPSproj}),
there is a sequence of prime intervals \eqref{Eq:sequence}
(also satisfying \eqref{E:geq}).
Since $\fp$ is on the upper boundary of $L$,
we cannot ``climb up'' from $\fp$;
it follows that $\fp = \fr$.
Therefore, $1_\fp = 1_{\fr} \geq 1_{\fq}$
by \eqref{E:geq}, contradicting our assumption 
that $1_\fp < 0_\fq \prec 1_\fq$.
\end{proof}

Now we are ready to prove our main result.
Let $t = \length[c_l, 1] + \length[c_r, 1]$.
By Theorem~\ref{T:prime}, 
the lattice $L$ has exactly~$t$ prime ideals,
and each prime ideal has an associated prime congruence,
a dual atom. So $\Con L$ has at least $t$ dual atoms.
By~Theorem~\ref{T:dualatom}, all dual atoms of $\Con L$
are prime congruences, 
so $\Con L$ has exactly $t$ dual atoms.

\section{Meet semidistributive lattices}\label{S:Meet}

A lattice $L$ is \emph{meet-semidistributive},
if the following implication holds:
\begin{equation}\label{E:xx}
x \mm y = x \mm z \text{ implies that } x \mm y = x \mm (y \jj z)\text{ for all } x,y,z \in L.\tag{SD$_\mm$}
\end{equation}

This implication was introduced
by  P.\,M. Whitman \cite{pW41} and \cite{pW42}
as a property of free lattices.
It also holds for SPS lattices.

\begin{lemma}\label{L:sd}
Let $L$ be an SPS lattice. 
Then the implication \lp SD$_\mm$\rp holds in $L$.
\end{lemma}

\begin{proof}
Assume that it does not hold. 
Then there are elements $a,b,c \in L$
such that $a \mm b = a \mm c$  
but $a \mm b \neq a \mm (b \jj c)$.
Then $x \neq y \in \set{a \mm (b \jj c), b, c}$
satisfy that $x \mm y = a \mm b$, 
so we cab choose elements $a',b',c' \in L$
so that $a \mm b \prec a' \leq a \mm (b \jj c)$, 
$a \mm b \prec b' \leq b$, $a \mm b \prec c' \leq c$,
contradicting Lemma~\ref{L:known}(i).
\end{proof}

For some references about semidistributive lattices, 
see K. Adaricheva, V.\,A. Gorbunov, 
V.\,I. Tumanov~\cite{AGT},
G.~Cz\'edli, L. Ozsv\'art, and B. Udvari \cite{COU}, 
S.\,P. Avann~\cite{sA68}, and
R.P. Dilworth~\cite{rD40}.

In the rest of this section, 
we outline the proof of the following variant 
of Theorem~\ref{T:main}.

\begin{named}{Theorem \ref{T:main}'} 
If $L$ is a finite meet-semidistributive lattice, 
then the meet of the dual atoms is the least congruence 
$\bgd$ with $L/\bgd$ distributive.
\end{named}

This result and its proof is due to Ralph Freese,
who emailed me after this paper was completed. 
Prossor Freese kindly suggested to me to
``feel free to use it in your paper''.

The following sketch of the proof (slightly edited) 
is from his email.

Since the class $\B D$ of distributive lattices 
is closed under subdirect products, we get the first statement.  

\begin{lemma}\label{L:1}
Every lattice L has a unique minimal congruence
$\bgd$ such that $L/\bgd$ is distributive. 
\end{lemma}

$L/\bgd$ is called the \emph{reflection} of $L$
into $\B D$.

\begin{lemma}\label{L:2}
For the congruence $\bgd$ of Lemma~\ref{L:1},
we have
\[
   \bgd = \MM \E C,
\]
where $\E C$ is the
set of those dual atoms of $\Con L$, 
whose corresponding quotient is $\SC 2$,
the two-element chain.
\end{lemma}

\begin{lemma}\label{L:3}
Every meet-semidistributive lattice with $0$ 
has $\SC 2$ as a homomorphic image.
\end{lemma}

We apply these lemmas to prove Theorem \ref{T:main}'. 
Since the lattice $L$ is finite 
and meet-semidistributive,
it follows that every homomorphic image 
of $L$ is also meet-semidistributive, 
and so every dual atom of $\Con L$ is in $\E C$.

\end{document}